\documentclass[reqno,12pt]{amsart}

\usepackage{amsmath,amssymb,amsfonts,amscd,amsthm}
\usepackage{mathrsfs,mathabx}
\usepackage{latexsym,graphicx,cite,cases,array,booktabs}
\usepackage{color}
\definecolor{rot}{rgb}{1.000, 0.000, 0.000}

\newtheorem{theorem}{Theorem}

\newtheorem{lemma}[theorem]{Lemma}

\newtheorem{remark}[theorem]{Remark}

\numberwithin{equation}{section}

\title[Stability for a Borg--Levinson theorem]{Stability for the multi-dimensional Borg--Levinson theorem of the biharmonic operator}

\author[P. Li]{Peijun Li}
\address{Department of Mathematics, Purdue University, West Lafayette, Indiana
47907, USA}
\email{lipeijun@math.purdue.edu}

\author[X. Yao]{Xiaohua Yao}
\address{School of Mathematics and Statistics, Central China Normal University,
Wuhan 430079, China}
\email{yaoxiaohua@mail.ccnu.edu.cn}

\author[Y. Zhao]{Yue Zhao}
\address{School of Mathematics and Statistics, Central China Normal University,
Wuhan 430079, China}
\email{zhaoyueccnu@163.com}



\begin{document}

\begin{abstract}

In this paper, we prove a conditional H\"older stability estimate for the inverse spectral problem of the biharmonic operator. The proof employs the resolvent estimate and a Weyl-type law for the biharmonic operator which were obtained by the authors 
in \cite{LYZ}. This work extends nontrivially the result in \cite{stefanov} from the second order Schr\"{o}dinger operator to the fourth order biharmonic operator. 

\end{abstract}

\maketitle

\section{Introduction}

The topic of meromorphic continuation of the outgoing resolvent and related resolvent estimates for elliptic operators is  central in scattering theory (see e.g. \cite{DJ18, Dy15a, Zw17}). 
Physically, the poles of the meromorphic continuation are closely related to the scattering resonances, which appear in many research areas of mathematics, physics, and engineering. We refer to the monograph \cite{Dyatlov} for a comprehensive introduction to mathematical theory of this subject. Recently, the stability estimates for the inverse source problems were obtained in \cite{LYZ, LZZ} by using the holomorphic domain and an upper bound for the resolvent of the elliptic operator. Another application can be found in \cite{Cakoni} for a study on the  duality between scattering poles and transmission eigenvalues in scattering theory. To further explore the applications of the scattering theory to other topics in the field of inverse problems, in this paper, we intend to study an inverse spectral problem for the biharmonic operator. The inverse spectral problem may be considered as an inverse boundary value problem. As a representative example, a fundamental work can be found in \cite{Uhlmann} on the Calder\'{o}n problem where the scattering theory played a crucial role. 

\vskip0.15cm

We briefly review the existing literature on the inverse spectral problem for the Schr\"{o}dinger operator. The classical one-dimensional inverse spectral problem was studied in \cite{borg, levinson}. A uniqueness result was established in \cite{NSU} for the multi-dimensional problem by representing the Dirichlet-to-Neumann (DtN) map in terms of the spectral data. The uniqueness of the inverse spectral problem with partial spectral data was discussed in \cite{Isozaki}.
For inverse spectral problems on Riemannian manifolds and in a periodic waveguide, we refer the reader to \cite{BK, Kurylev, KKL, Kian}. Stability of the inverse spectral problems was addressed in \cite{AS, stefanov}. Recent developments on numerical methods can be found in \cite{BXZ, XZ} for the one-dimensional inverse spectral problems. 

\vskip0.15cm

Since there is already a vast amount of literature on the inverse spectral problems for the Schr\"{o}dinger operator, we wish to extend the results to higher order elliptic operators. The inverse problems of biharmonic operators have significant applications in various areas including the theory of vibration of beams, the hinged plate configurations and the scattering by grating stacks \cite{GGS, MMMP}. We refer the reader to \cite{ikehata, isakov} for some uniqueness results of the inverse problems of higher order elliptic operators. In \cite{katya}, the uniqueness with full or incomplete spectral data was studied for the elliptic operators of higher order with constant coefficients. However, to the best of our knowledge, there is no stability estimate so far for the inverse spectral problem of the elliptic operators of higher order. 

\vskip0.15cm

This work is motivated by \cite{AS, Isozaki, stefanov}, which were concerned with the inverse spectral problem of determining the potential function of the Schr\"odinger operator from the spectral data consisting of the eigenvalues and normal derivatives of the eigenfunctions on the boundary. In \cite{Isozaki}, the author showed that even if a finite number of spectral data is unavailable, the potential can still be uniquely determined. The proof utilized an idea of the Born approximation in scattering theory. 
A stability theorem for the inverse spectral problem was obtained in \cite{AS} by using partial spectral data. The approach was to connect the hyperbolic DtN map associated with a hyperbolic equation with the DtN map of the stationary Schr\"odinger operator. The proof of the stability estimate was built upon \cite{Rakesh}, which studied an inverse problem for the wave equation by hyperbolic DtN map. Based on \cite{AS, Isozaki}, the authors proved in \cite{stefanov} the uniqueness result \cite[Theorem 2.1]{stefanov} by assuming that the spectral data are only known asymptotically for the Schr\"odinger operator. Moreover, a H\"older stability estimate was obtained in \cite[Theorem 2.2]{stefanov}, which assumes that a finite number of spectral data is not available. The proof of \cite[Theorem 2.2]{stefanov} combines the crucial integral identity introduced in \cite[Lemma 2.2]{Isozaki} and the method used in \cite{AS}. We also point out that the proofs in \cite{Isozaki, stefanov} rely on the resolvent estimate for the Schr\"odinger operator and a Weyl-type law is crucial in the proof of the stability estimate.

\vskip0.15cm

Recently, we proved an increasing stability estimate for the inverse source problem of the biharmonic operator \cite{LYZ}. Meanwhile, we obtained the resolvent estimate and a Weyl-type inequality for the biharmonic operator. As a consequence, we hope to extend the results in \cite{AS, Isozaki, stefanov} from the Schr\"odinger operator to the biharmonic operator. Clearly, the extension is nontrivial. Compared with the elliptic operators of second order, the biharmonic operator is more sophisticated. For instance, it is required to investigate two sets of the DtN maps and use more spectral data in order to study the inverse problems of the biharmonic operator. Moreover, the resolvent set and resolvent estimate of the biharmonic operator differ significantly from 
the Schr\"odinger operator. As pointed out in \cite{May}, the methods for the second order equations may not work for 
higher-order equations. The solutions of higher-order equations have more complicated properties. In this work, we prove a conditional H\"{o}lder stability for the inverse spectral problem of the biharmonic operator. The proof is based on a combination of an Isozaki's representation formula (cf. Lemma \ref{Isozaki_bi}) and a Weyl-type law of the Dirichlet eigenvalue problem for the biharmonic operator with a potential (cf. Lemma \ref{eigenfunction_est1}). 

\vskip0.15cm

Next we introduce some notations and state the main result of this paper. 

\vskip0.15cm

Let $B_R = \{x\in\mathbb R^n ~:~ |x|< R\}$, where $n\geq 3$ is odd and $R>0$ is a constant. Denote by $\partial B_R$ the boundary of $B_R$. We consider the eigenvalue problem with the Navier boundary condition
\[
\begin{cases}
(\Delta^2 + V) \phi_k=\lambda_k\phi_k&\quad\text{in }B_R,\\
\Delta \phi_k = \phi_k=0&\quad\text{on }\partial B_R,
\end{cases}
\]
where $\{\lambda_j, \phi_j\}_{j=1}^\infty$ denotes the positive increasing eigenvalues and orthonormal eigenfunctions. 

\vskip0.15cm

Hereafter, the notation $a\lesssim b$ stands for $a\leq Cb,$ where $C>0$ is a generic constant which may change step by step in the proofs. The following Weyl-type law for the biharmonic operator with a potential given in Lemma \ref{eigenfunction_est1} is crucial in the proof of the stability:
\begin{align}\label{weyl}
 |\lambda_k| \sim k^{4/n}, \quad \|\partial_\nu\phi_k\|_{L^2(\partial B_R)}\lesssim k^{2/n}, \quad \|\partial_\nu(\Delta\phi_k)\|_{L^2(\partial B_R)}\lesssim k^{4/n},
\end{align}
where $\nu$ is the unit outward normal vector to $\partial B_R$.  We mention that the Weyl-type law \eqref{weyl} for the biharmonic operator was proved in \cite{LYZ} by using an argument of commutator, which would yield a sharper result than using only the standard elliptic regularity theory for the Schr\"odinger operator \cite[Lemma 2.5]{AS}. 
Consequently, this sharper Weyl-type law \eqref{weyl} leads to a better stability estimate for the inverse spectral problem of the biharmonic operator.

Consider an integer $m$ such that
\[
m>n/4 + 1.
\]
It follows from \eqref{weyl} that both the series 
\[
\sum_{k\geq 1}k^{-4m/n} \|\partial_\nu\phi_k\|_{L^2(\partial B_R)} \quad \text{and}\quad \sum_{k\geq 1}k^{-4m/n} \|\partial_\nu(\Delta\phi_k)\|_{L^2(\partial B_R)}
\] 
converge absolutely in $L^2(\partial B_R)$. 

\vskip0.15cm

For two potential functions $V_1, V_2\in L^\infty(B_R)$, we denote the positive increasing eigenvalues and orthonormal eigenfunctions of $V_1$ and $V_2$
by $\{\lambda^{(1)}_j, \phi^{(1)}_j\}_{j=1}^\infty$ and $\{\lambda^{(2)}_j, \phi^{(2)}_j\}_{j=1}^\infty$, respectively. Let $E\geq 0$ be any fixed integer and define
the spectral data discrepancy by
\begin{align*}
\varepsilon_0 &= \max_{k\geq 1} |\lambda^{(1)}_{k + E} - \lambda^{(2)}_{k + E}|,\\
\varepsilon_1 &= \sum_{k\geq 1} k^{-4m/n} \|\partial_\nu\phi^{(1)}_{k + E} - \partial_\nu\phi^{(2)}_{k + E}\|_{L^2(\partial B_R)},\\
\varepsilon_2 &= \sum_{k\geq 1}k^{-4m/n} \|\partial_\nu(\Delta\phi^{(1)}_{k + E}) - \partial_\nu(\Delta\phi^{(2)}_{k + E})\|_{L^2(\partial B_R)}.
\end{align*}

\vskip0.15cm

The following theorem concerns the stability of the inverse problem and is the main result of the paper.

\begin{theorem}\label{main}
For $V_1, V_2\in L^\infty(B_R)$ satisfying $V:= V_1 - V_2\in H_0^1(B_R)$ and 
\[
\|V_1\|_{L^\infty(B_R)} + \|V_2\|_{L^\infty(B_R)} + \|V\|_{H_0^1(B_R)} \leq Q,
\]
there exist two constants $C = C(m, Q, n)$ and $0<\delta<1$ such that
\begin{align}\label{stability}
\|V_1 - V_2\|_{L^2(B_R)}\leq C\varepsilon^\delta,
\end{align}
where $\varepsilon = \varepsilon_0 + \varepsilon_1+ \varepsilon_2$.
\end{theorem}

The assumption $V:= V_1 - V_2\in H_0^1(B_R)$ will be used to control the high frequency tail of the Fourier transform of $V$.
This is a commonly used argument in the study of the inverse problems (cf. \cite[Proof of Proposition 1]{Ales}, \cite[$(4.3)$]{Isakov1}, \cite{LYZ}).

\vskip0.15cm

The above result extends \cite[Theorem 2.2]{stefanov} from the Schr\"odinger operator to the biharmonic operator. It can be seen from \eqref{stability} that even if a finite number of spectral data is not available, the conditional H\"{o}lder stability can still be obtained, which clearly implies the uniqueness of the inverse spectral problem. Compared with \cite[Theorem 2.2]{stefanov}, the analysis of the biharmonic operator is more involved. Specifically, it is required to investigate two sets of the DtN maps and use more spectral data in order to study the inverse problems of the biharmonic operator. As a result, we must extend the crucial integral identity presented in \cite[Lemma 2.2]{Isozaki} and several important lemmas proved in \cite{AS} from the Schr\"odinger operator to the biharmonic operator. The extensions require the  Weyl-type inequality \eqref{weyl} and the resolvent estimate for the biharmonic operator which were proved in \cite{LYZ}.

\vskip0.15cm

The paper is organized as follows. The two sets of DtN maps are introduced in Section \ref{DtN maps}. 
Section \ref{proof} is devoted to the proof of the stability. In Appendix, we present the estimates of the resolvent and a Weyl-type law for the biharmonic operator.

\section{The DtN maps}\label{DtN maps}

In this section, we consider two families of the DtN maps and study their mapping properties. Let $V\in L^\infty(B_R)$ and $\lambda\notin \{\lambda_k\}_{k=1}^\infty$. 
Given any $f\in H^{3/2}(\partial B_R)$ and $g\in H^{-1/2}(\partial B_R)$, consider the boundary value problem
\begin{align}\label{eqn}
\begin{cases}
H_V u - \lambda u = 0 &\quad \text{in}\, B_R,\\
u = f &\quad \text{on}\,\partial B_R,\\
\Delta u = g &\quad \text{on}\,\partial B_R,
\end{cases}
\end{align}
where $H_V=\Delta^2+V$. Clearly, it has a unique weak solution $u\in H^2(B_R)$. We introduce two DtN maps
\begin{align*}
\Lambda_1(\lambda)&: f \rightarrow \partial_\nu u\vert_{\partial B_R},\\
\Lambda_2(\lambda)&: g\rightarrow \partial_\nu (\Delta u)\vert_{\partial B_R}, 
\end{align*}
where $\Lambda_1(\lambda)$ and $\Lambda_2(\lambda)$ define bounded operators from $H^{3/2}(\partial B_R)$ to $H^{1/2}(\partial B_R)$ and from $H^{-1/2}(\partial B_R)$ to $H^{-3/2}(\partial B_R)$, respectively. 

\vskip0.15cm

Next, we derive formal representations of $\Lambda_1(\lambda)$ and $\Lambda_2(\lambda)$ by using the spectral data.
Multiplying both sides of \eqref{eqn} by $\phi_k$ and using the integration by parts, we have 
\[
\int_{B_R} u \phi_k {\rm d}x =  \frac{1}{\lambda_k - \lambda}  
\Big(\int_{\partial B_R} \partial_\nu\phi_k f {\rm d}s(y) + \int_{\partial B_R} \partial_\nu(\Delta\phi_k)g{\rm d}s(y)\Big),
\]
which formally gives
\[
u(x, \lambda) = \sum_{k=1}^\infty \phi_k \frac{1}{\lambda_k - \lambda}  
\Big(\int_{\partial B_R} \partial_\nu\phi_k f{\rm d}s(y) + \int_{\partial B_R} \partial_\nu(\Delta\phi_k)g{\rm d}s(y)\Big), \quad x\in B_R.
\]
Thus, for $\lambda\notin \{\lambda_k\}_{k=1}^\infty$, the DtN maps can be represented by 
\begin{align*}
\Lambda_1(\lambda)(f, g) &= \sum_{k = 1}^\infty \partial_\nu\phi_k\Big\vert_{\partial B_R} \frac{1}{\lambda_k - \lambda} 
\Big(\int_{\partial B_R} \partial_\nu\phi_k f{\rm d}s(y) + \int_{\partial B_R} \partial_\nu(\Delta\phi_k)g{\rm d}s(y)\Big)
\end{align*}
and
\begin{align*}
\Lambda_2 (\lambda)(f, g) &= \sum_{k = 1}^\infty \partial_\nu(\Delta\phi_k)\Big\vert_{\partial B_R} \frac{1}{\lambda_k - \lambda}
\Big( \int_{\partial B_R} \partial_\nu\phi_k f{\rm d}s(y) + \int_{\partial B_R} \partial_\nu(\Delta\phi_k)g{\rm d}s(y)\Big).
 \end{align*}
 
 \vskip0.15cm
 
 However, the series on the right hand side may not converge absolutely. It was shown in \cite[Lemma 2.6]{AS} that some higher order formal derivatives converge absolutely. Let 
 \begin{align*}
 \Lambda_1^{(m)}(\lambda):= \frac{{\rm d}^m}{{\rm d}\lambda^m} \Lambda_1(\lambda),\quad 
 \Lambda_2^{(m)}(\lambda):= \frac{{\rm d}^m}{{\rm d}\lambda^m} \Lambda_2(\lambda).
 \end{align*}
 By the Weyl-type law \eqref{weyl}, for $m\gg 1$, the above two series converge absolutely.  Precisely, we have the following lemma.
 
 \begin{lemma}\label{derivative}
 For $m>n/4 + 1$ and $\lambda\notin \{\lambda_k\}_{k=1}^\infty$, the series
 \begin{align*}
 \Lambda^{(m)}_1(\lambda)(f,g)&= -m!\sum_{k = 1}^\infty \partial_\nu\phi_k\Big\vert_{\partial B_R} \frac{1}{(\lambda_k - \lambda)^{m + 1}} 
\Big(\int_{\partial B_R} \partial_\nu\phi_k f{\rm d}s(y) \\
&\quad + \int_{\partial B_R} \partial_\nu(\Delta\phi_k)g{\rm d}s(y)\Big)
\end{align*}
and
\begin{align*}
\Lambda^{(m)}_2(\lambda)(f,g)&= -m!\sum_{k = 1}^\infty \partial_\nu(\Delta\phi_k)\Big\vert_{\partial B_R} \frac{1}{(\lambda_k - \lambda)^{m + 1}}
\Big( \int_{\partial B_R} \partial_\nu\phi_k f{\rm d}s(y) \\
&\quad+ \int_{\partial B_R} \partial_\nu(\Delta\phi_k )g{\rm d}s(y)\Big),
 \end{align*}
converge absolutely in $H^{1/2}(\partial B_R)$ and $H^{-3/2}(\partial B_R)$, respectively. 
Moreover, $\Lambda^{(m)}_1(\lambda)$ and $\Lambda^{(m)}_2(\lambda)$ can be extended to meromorphic families with poles at the eigenvalues.
 \end{lemma}
 
Denote the DtN maps of $V_\alpha$ by $\Lambda_{\alpha, 1}, \Lambda_{\alpha, 2}, \alpha = 1, 2,$ respectively. The following lemma gives the mapping properties of the derivatives of the DtN maps. The proof is motivated by \cite[Lemma 2.32]{choulli} which is dated back to \cite[Lemma 2.3]{AS}. The lemma extends the result from the Laplacian operator to the biharmonic operator.
 
 \begin{lemma}\label{ddtn}
  Assume that $\lambda\notin \{\lambda^{(1)}_k\}_{k=1}^\infty \cup \{\lambda^{(2)}_k\}_{k=1}^\infty$ and let $l$ be a positive integer. The following estimates hold: 
\begin{align*}
\|\Lambda_{1,1}^{(j)}(\lambda) - \Lambda_{2,1}^{(j)}(\lambda)\|_{\mathcal{L}(H^{\frac{3}{2}}(\partial B_R), \, H^{t_1}(\partial B_R))} &\lesssim \frac{1}{|\lambda|^{j + \sigma_1}},\\
\|\Lambda_{1, 2}^{(j)}(\lambda) - \Lambda_{2, 2}^{(j)}(\lambda)\|_{\mathcal{L}(H^{-\frac{1}{2}}(\partial B_R), \, H^{t_2}(\partial B_R))} &\lesssim \frac{1}{|\lambda|^{j + \sigma_2}},
\end{align*}
where $0\leq j\leq l$, $|\lambda|\geq 2Q$, and
\begin{align*}
\sigma_1 = \frac{1 - 2t_1}{4}, \quad - \frac{3}{2}\leq t_1 \leq \frac{1}{2}, \quad
\sigma_2 = \frac{-3 - 2t_2}{4}, \quad -\frac{7}{2}\leq t_2 \leq -\frac{3}{2}.
\end{align*}
\end{lemma}

\begin{proof}

Let $f\in H^{3/2}(\partial B_R)$, $g\in H^{-1/2}(\partial B_R)$ and $u_j, j = 1, 2$ be the solution to the boundary value problem
\begin{align*}
\begin{cases}
\Delta^2 u_j + V_j u_j - \lambda u_j = 0 &\quad \text{in}\, B_R,\\
u_j = f &\quad \text{on}\, \partial B_R,\\
\Delta u_j = g &\quad \text{on}\, \partial B_R.
\end{cases}
\end{align*}
Let $u := u_1 - u_2$. A simple calculation yields 
\begin{align*}
\begin{cases}
\Delta^2 u + V_1 u - \lambda u = (V_2 - V_1)u_2 &\quad \text{in}\, B_R,\\
u = 0 &\quad \text{on}\, \partial B_R,\\
\Delta u = 0 &\quad \text{on}\, \partial B_R.
\end{cases}
\end{align*}

For $|\lambda|\geq 2Q$, multiplying both sides of the above equation by $u$ and integrating by parts, we obtain 
\begin{align}\label{2.61}
\|u\|_{L^2(B_R)}\lesssim \frac{1}{|\lambda|} \|u_2\|_{L^2(B_R)}.
\end{align}
It follows from Theorem \ref{regularity} that 
\[
\|u_2\|_{L^2(B_R)} \lesssim \|f\|_{H^{3/2}(\partial B_R)} + \|g\|_{H^{-1/2}(\partial B_R)},
\]
which gives 
\begin{align}\label{2.62}
\|u\|_{L^2(B_R)}\lesssim \frac{1}{|\lambda|} \big(\|f\|_{H^{3/2}(\partial B_R)} + \|g\|_{H^{-1/2}(\partial B_R)}\big).
\end{align}

 Denote by $u^\prime(\lambda)$ and $u_j^\prime(\lambda)$ the derivatives of $u(\lambda)$ and $u_j(\lambda)$ with respect to $\lambda$. It can be verified that $u_2^\prime(\lambda)$
 satisfies
 \begin{align*}
\begin{cases}
\Delta^2 u_2^\prime(\lambda) + V_2 u_2^\prime(\lambda) - \lambda u_2^\prime(\lambda) = u_2 &\quad \text{in}\, B_R,\\
u_2^\prime(\lambda) = 0 &\quad \text{on}\, \partial B_R,\\
\Delta u_2^\prime(\lambda) = 0 &\quad \text{on}\, \partial B_R.
\end{cases}
\end{align*}
Using similar arguments as \eqref{2.61}, we get
\begin{align}\label{2.63}
\|u_2^\prime(\lambda)\|_{L^2(B_R)}\lesssim \frac{1}{|\lambda|} \big(\|f\|_{H^{3/2}(\partial B_R)} + \|g\|_{H^{-1/2}(\partial B_R)}\big).
\end{align}
Since $u^\prime(\lambda)$ satisfies
\begin{align*}
\begin{cases}
\Delta^2 u^\prime(\lambda) + V_1 u^\prime(\lambda) - \lambda u^\prime(\lambda) = u(\lambda) + (V_2 - V_1)u^\prime_2(\lambda) &\quad \text{in}\, B_R,\\
u^\prime(\lambda) = 0 &\quad \text{on}\,\partial B_R,\\
\Delta u^\prime(\lambda) = 0 &\quad \text{on}\,\partial B_R,
\end{cases}
\end{align*}
we have
\[
\|u^\prime(\lambda)\|_{L^2(B_R)} \lesssim \frac{1}{|\lambda|} \| u(\lambda) + (V_2 - V_1)u^\prime_2(\lambda)\|_{L^2(B_R)}.
\]
Combining \eqref{2.62} and \eqref{2.63} leads to 
\begin{align}\label{2.64}
\|u^\prime(\lambda)\|_{L^2(B_R)} \lesssim \frac{1}{|\lambda|^2} (\|f\|_{H^{3/2}(\partial B_R)} + \|g\|_{H^{-1/2}(\partial B_R)}).
\end{align}

On the other hand,  it follows from the standard regularity results of elliptic equations that 
\begin{align*}
\|u^\prime(\lambda)\|_{H^4(B_R)} \lesssim |\lambda| \|u^\prime(\lambda)\|_{L^2(B_R)} + \|u(\lambda)\|_{L^2(B_R)}
+ \|u_2^\prime(\lambda)\|_{L^2(B_R)},
\end{align*}
which gives
\begin{align}\label{2.65}
\|u^\prime(\lambda)\|_{H^2(B_R)} \lesssim \frac{1}{|\lambda|} \big(\|f\|_{H^{3/2}(\partial B_R)} + \|g\|_{H^{-1/2}(\partial B_R)}\big).
\end{align}
Recalling the interpolation inequality 
\[
\|w\|_{H^s(B_R)} \lesssim \|w\|^{1 - s/2}_{L^2(B_R)} \|w\|^{s/2}_{H^2(B_R)}, \quad 0\leq s\leq 2, \,  w\in H^2_0(B_R),
\]
we obtain from \eqref{2.64}--\eqref{2.65} that 
\begin{align*}
\|u^\prime(\lambda)\|_{H^s(B_R)} \lesssim \frac{1}{|\lambda|^{2 - s/2}}\big(\|f\|_{H^{3/2}(\partial B_R)} + \|g\|_{H^{-1/2}(\partial B_R)}\big), \quad 0\leq s\leq 2.
\end{align*}
Therefore, we have
\begin{align*}
\|\partial_\nu u^\prime(\lambda)\|_{H^{s - 3/2}(B_R)} &\lesssim \frac{1}{|\lambda|^{2 - s/2}}\big(\|f\|_{H^{3/2}(\partial B_R)} + \|g\|_{H^{-1/2}(\partial B_R)}\big), 
\quad 0\leq s\leq 2,
\end{align*}
and
\begin{align*}
\|\partial_\nu (\Delta u^\prime(\lambda))\|_{H^{s - 7/2}(B_R)} &\lesssim \frac{1}{|\lambda|^{2 - s/2}}\big(\|f\|_{H^{3/2}(\partial B_R)} + \|g\|_{H^{-1/2}(\partial B_R)}\big),
\quad 0\leq s\leq 2,
\end{align*}
which completes the proof by letting $t_1 = s - 3/2, t_2 = s - 7/2$ and an application of induction.
\end{proof}

\section{Proof of the main result}\label{proof}

First we show an Isozaki's representation formula which links the potential function and the spectral data. A similar formula may be found in \cite[Lemma 2.2]{Isozaki} for the Schr\"{o}dinger operator. The result is closely related to the scattering theory. Specifically, 
let $\varphi_{\omega}(x) = e^{{\rm i} \sqrt[4]{\lambda}\omega\cdot x}$ for $\lambda\in\mathbb{C}\backslash (-\infty, 0)$ with $\Im \sqrt[4]{\lambda}\geq 0,$ which may be considered as an incident plane wave with direction $\omega$ and wavenumber $\sqrt[4]{\lambda}$. Denote by $R_V(\lambda) = (H_V - \lambda)^{-1}$ the resolvent of $H_V$. Let $\Omega_\delta$ be the holomorphic domain of the resolvent $R_V(\lambda)$ obtained in Theorem \ref{bound_2}.

Define
 \[
 S(\omega, \theta) = -\sqrt{\lambda}  \int_{\partial B_R} \Lambda_1(\lambda)(\varphi_{\omega})\varphi_{-\theta}
 + \Lambda_2(\lambda) (\varphi_{\omega}) \varphi_{-\theta}{\rm d}s(x), \quad \omega, \theta\in\mathbb S^{n - 1},
 \]
which may be regarded as the scattering matrix for the case of the biharmonic operator.

\begin{lemma}\label{Isozaki_bi}
Assume that $\lambda\notin \{\lambda_k\}_{k=1}^\infty$.
For $\lambda \in \Omega_\delta$ it holds that 
\begin{align*}
S(\omega, \theta) = -\int_{B_R} V e^{{\rm i} \sqrt[4]{\lambda}(\omega - \theta)\cdot x} {\rm d}x - \int_{B_R} R_V(\lambda) (-V \varphi_{\omega}) V\varphi_{-\theta}{\rm d}x
- 2\sqrt{\lambda} \int_{\partial B_R} \varphi_{\omega} \partial_\nu (\varphi_{-\theta}) {\rm d}s(x).
\end{align*}
\end{lemma}

\begin{proof}
Consider the boundary value problem
\begin{align*}
\begin{cases}
H_V u - \lambda u=0  &\quad \text{in}\, B_R,\\
u = 0 &\quad \text{on}\, \partial B_R,\\
\Delta u = 0  &\quad \text{on}\, \partial B_R,
\end{cases}
\end{align*}
which has a unique trivial solution $u = 0$. Decompose $u$ as $u = \tilde{u} + \varphi_{\omega}$. Then we have 
$\tilde{u} = R_V(\lambda)(-V\varphi_{\omega})$. Moreover, $\tilde{u}$ satisfies the boundary value problem
\begin{align*}
\begin{cases}
H_V\tilde{u} - \lambda \tilde{u} = -V\varphi_{\omega} &\quad \text{in}\,B_R,\\
\tilde{u} = -\varphi_{\omega} &\quad \text{on}\,\partial B_R,\\
\Delta \tilde{u} = -\Delta \varphi_{\omega}  &\quad \text{on}\,\partial B_R.
\end{cases}
\end{align*}
 Multiplying both sides of the above equation by $\varphi_{-\theta}$ and integrating by parts yield
 \begin{align*}
 &\int_{\partial B_R} \partial_\nu(\Delta\tilde{u}) \varphi_{-\theta} {\rm d}s(x)+ \int_{\partial B_R} \partial_\nu\tilde{u} \Delta\varphi_{-\theta}{\rm d}s(x)\\
 &= -\int_{B_R}\big( V \tilde{u} \varphi_{-\theta} + V \varphi_{\omega} \varphi_{-\theta}\big) {\rm d}x + \int_{\partial B_R} \big(\Delta\tilde{u} \partial_\nu(\varphi_{-\theta}) + \tilde{u} \partial_\nu(\Delta(\varphi_{-\theta}))\big) {\rm d}s(x),
 \end{align*}
 which completes the proof.
\end{proof}

Define 
\begin{align*}
\theta = c\eta + \frac{1}{2\zeta}\xi,\quad
\omega =c \eta - \frac{1}{2\zeta}\xi, \quad
\sqrt[4]{\lambda} = \zeta + {\rm i},
\end{align*} 
where the constant $c$ is chosen such that $\theta , \omega\in\mathbb S^{n - 1}$. Compared with \cite{Isozaki}, the difference comes from the fourth root of $\lambda$ instead of the square root of $\lambda$ due to the nature of the biharmonic operator. Denote by $S_\alpha(\omega, \theta)$ the above defined function $S$ corresponding to $V_\alpha$, where $ \alpha= 1, 2$.
Using the resolvent estimate in Theorem \ref{bound_2} 
\begin{align}\label{resolvent}
\|R_V(\lambda)\|_{\mathcal{L}(L^2(B_R))} \lesssim \frac{1}{\sqrt{|\lambda|}}, \quad \lambda\in\Omega_\delta,
\end{align}
and a common technique to control the high frequency tail, we obtain the following lemma, which is  useful in the proof of the main theorem. A similar procedure is also used in \cite{stefanov} for the Schr\"{o}dinger operator.

\begin{lemma}\label{control}
There exists $\zeta_0>1$ sufficiently large such that for $\zeta\geq\zeta_0$ 
\begin{align*}
\|V_1 - V_2\|^2_{L^2(B_R)}\lesssim \frac{1}{\zeta^{\frac{1}{n}}} + |\zeta|^{1/2} |S_1(\omega, \theta) - S_2(\omega, \theta)|^2.
\end{align*}
\end{lemma}

\begin{proof}

Denote the difference of the two unknown potentials by $V = V_1 - V_2$.
Recall that $\sqrt[4]{\lambda} = \zeta + {\rm i}$ with $\zeta\geq 1$. By the integral identity in Lemma \ref{Isozaki_bi} and the resolvent estimate \eqref{resolvent}, we obtain 

\begin{align}\label{relation}
|\hat{V} (\xi + \frac{\rm i}{\zeta}\xi)| \lesssim \frac{1}{\zeta^2} + |S_1(\omega, \theta) - S_2(\omega, \theta)|.
\end{align}
Let $f(t) = \hat{V} (\xi + \frac{{\rm i} t}{\zeta}\xi)$. A simple calculation yields   
\begin{align}\label{fv}
|\hat{V} (\xi)| &= \Big| \int_0^1 f^\prime(t){\rm d}t - f(1)\Big|\notag\\
&\leq |\hat{V}(\xi + \frac{\rm i}{\zeta}\xi)| + \frac{1}{\zeta} \sup_{0\leq t\leq 1} |\nabla \hat{V}(\xi + \frac{{\rm i }t}{\zeta}\xi)\cdot\xi|.
\end{align}
It follows from Fourier transform that 
\[
|\partial_{\xi_i} \hat{V}(\xi + \frac{{\rm i }t}{\zeta}\xi)| = |\widehat{x_i V}(\xi + \frac{{\rm i }t}{\zeta}\xi)|\lesssim e^{\frac{|\xi|}{\zeta}}\|V\|_{L^\infty(B_R)},
\quad 0<t<1,
\]
which, along with \eqref{fv}, gives  
\begin{align}\label{5.1}
|\hat{V}(\xi)| \lesssim |\hat{V}(\xi + \frac{\rm i}{\zeta}\xi)|  + \frac{|\xi|}{\zeta} e^{\frac{|\xi|}{\zeta}} \|V\|_{L^\infty(B_R)}.
\end{align}

Combining \eqref{relation} and \eqref{5.1}, we have 
\[
|\hat{V}(\xi)| \lesssim \frac{1}{\zeta^2} + \frac{|\xi|}{\zeta} e^{\frac{|\xi|}{\zeta}} + |S_1(\omega, \theta) - S_2(\omega, \theta)|.
\]
It follows from taking integration of the above inequality in the domain $|\xi|\leq \zeta^{1/(2n)}$ that 
\begin{align}\label{5.2}
\int_{|\xi|\leq\zeta^{1/(2n)}} |\hat{V}(\xi)|^2 {\rm d}\xi \lesssim \frac{1}{\zeta^{7/2}} + \frac{\zeta^{(2 + n)/(2n)}}{\zeta^2} e^{\zeta^{1/2n-1}}
+ |\zeta|^{1/2} |S_1(\omega, \theta) - S_2(\omega, \theta)|^2.
\end{align}
On the other hand, since $V\in H_0^1(B_R)$, we have the following inequality where the high frequency tail of $\hat{V}(\xi)$ is bounded by
the $H^1$ norm of $V$: 
\begin{align*}
\|\hat{V}\|^2_{L^2(\mathbb R^3)} &= \int_{\mathbb R^3} |\hat{V}(\xi)|^2 {\rm d}\xi\\
&= \int_{|\xi|\leq\zeta^{1/(2n)}}  |\hat{V}(\xi)|^2 {\rm d}\xi + \int_{|\xi|>\zeta^{1/(2n)}}  |\hat{V}(\xi)|^2 {\rm d}\xi\\
&\leq \int_{|\xi|\leq\zeta^{1/(2n)}}  |\hat{V}(\xi)|^2 {\rm d}\xi + \frac{1}{\zeta^{1/n}}\int_{|\xi|>\zeta^{1/(2n)}}  |\xi|^2 |\hat{V}(\xi)|^2 {\rm d}\xi\\
&\leq \int_{|\xi|\leq\zeta^{1/(2n)}}  |\hat{V}(\xi)|^2 {\rm d}\xi + \frac{1}{\zeta^{1/n}}\int_{\mathbb R^3}  (1 + |\xi|^2) |\hat{V}(\xi)|^2 {\rm d}\xi\\
&\leq \int_{|\xi|\leq\zeta^{1/(2n)}}  |\hat{V}(\xi)|^2 {\rm d}\xi + \frac{1}{\zeta^{1/n}} \|V\|^2_{H^1(\mathbb R^3)}.
\end{align*}
Since $\|V\|_{H^1(\mathbb R^3)}\leq Q$, by \eqref{5.2} and $\|V\|^2_{L^2(B_R)} = \|\hat{V}\|^2_{L^2(\mathbb R^3)}$, we obtain
\begin{align*}
\|V\|^2_{L^2(B_R)}\lesssim \frac{1}{\zeta^{\frac{1}{n}}} + |\zeta|^{1/2} |S_1(\omega, \theta) - S_2(\omega, \theta)|^2, 
\end{align*}
which completes the proof.
\end{proof}

Below we show the main theorem. Motivated by \cite[Theorem 2.2]{stefanov}, the proof employs the techniques of Taylor's formula and truncation of the DtN maps which were introduced in \cite[Proof of Proposition 2.1]{AS}.
It is worth mentioning that our result is not a direct consequence of \cite[Theorem 2.2]{stefanov} and the proof is more involved, since we have to deal with the more sophisticated biharmonic operator, and consider two sets of the DtN maps and spectral data. Moreover, the resolvent and the Weyl-type law for the biharmonic operator differ significantly from the Schr\"odinger operator. 

\vskip0.15cm

\begin{proof}
Throughout the proof, we assume that
$\lambda\in\Omega^{(1)}_\delta\cap\Omega^{(2)}_\delta$, where 
$\Omega^{(\alpha)}_\delta$ denotes the holomorphic domain obtained in Theorem \ref{bound_2} for the resolvent of $H_{V_\alpha}, \alpha = 1, 2,$
$\Re\lambda\notin \{\lambda^{(1)}_k\}_{k=1}^\infty \cup \{\lambda^{(2)}_k\}_{k=1}^\infty$ with $\Re\lambda>0$,
and $\Im\lambda\geq 1$. This assumption is allowed due to the definition of $\Omega_\delta$ in Theorem \ref{bound_2}, which contains the first quadrant of the complex plane.

Let $V = V_1 - V_2$. By Lemma \ref{control}, for $\zeta\geq\zeta_0$ where $\zeta_0>1$ is sufficiently large, we have 
\begin{align*}
\|V\|^2_{L^2(B_R)}\lesssim \frac{1}{\zeta^{\frac{1}{n}}} + |\zeta|^{1/2} |S_1(\omega, \theta) - S_2(\omega, \theta)|^2.
\end{align*}

Next we estimate $|S_1(\omega, \theta) - S_2(\omega, \theta)|^2$ by the two sets of DtN maps $\|\Lambda_{1, 1}(\lambda) - \Lambda_{2, 1}( \lambda)\|_1$ and $\|\Lambda_{1, 2}(\lambda) - \Lambda_{2, 2}(\lambda)\|_2$, where $\|\cdot\|_1$ and $\|\cdot\|_2$ stand for the norms in $\mathcal{L}(H^{3/2}(\partial B_R), L^2(\partial B_R))$ and $\mathcal{L}(H^{-1/2}(\partial B_R), H^{-3/2}(\partial B_R))$, respectively,  by choosing $t_1 = 0 $ and $t_2 = -\frac{3}{2}$ in Lemma \ref{ddtn}.  
Using the estimates
\[
\|\varphi_{\omega}\|_{H^{3/2}(\partial B_R)}\lesssim \zeta^{3/2}, \quad \|\varphi_{\omega}\|_{H^{-1/2}(\partial B_R)}\leq C, 
\quad \|\varphi_{-\theta}\|_{H^{3/2}(\partial B_R)}\leq \zeta^{3/2},
\]
one has
\begin{align*}
|S_1(\omega, \theta) - S_2(\omega, \theta)|\lesssim \zeta^{7/2} \Big(\|\Lambda_{1,1}(\lambda) - \Lambda_{2,1}(\lambda)\|_1 +  \|\Lambda_{1,2}(\lambda) - 
\Lambda_{2,2}(\lambda)\|_2 \Big).
\end{align*}
Then we get from Lemma \ref{control} that
\begin{align}\label{5.4}
\|V\|^2_{L^2(B_R)} \lesssim \frac{1}{\zeta^{\frac{1}{n}}} + \zeta^{15/2} \Big(\|\Lambda_{1,1}(\lambda) - \Lambda_{2,1}(\lambda)\|^2_1
+ \|\Lambda_{1,2}(\lambda) - \Lambda_{2,2}(\lambda)\|^2_2\Big).
\end{align}

In what follows we study $\Lambda_{\alpha, 1}( \lambda), \alpha= 1, 2$. We fix a positive integer $E$ and
decompose $\Lambda_{\alpha, 1}(\lambda)$ and $\Lambda_{\alpha, 2}(\lambda)$ into a sum of a finite series and an infinite one
as follows: 
\begin{align*}
\Lambda_{\alpha, 1} (\lambda) = \tilde{\Lambda}_{\alpha, 1} (\lambda) + \hat{\Lambda}_{\alpha, 1} (\lambda),\\
\Lambda_{\alpha, 2}( \lambda) = \tilde{\Lambda}_{\alpha, 2} (\lambda) + \hat{\Lambda}_{\alpha, 2}( \lambda),
\end{align*}
where 
\begin{align*}
\tilde{\Lambda}_{\alpha, 1} (\lambda)(f, g) &= \sum_{k>E} \partial_\nu\phi^{(\alpha)}_k\Big\vert_{\partial B_R} \frac{1}{\lambda^{(\alpha)}_k - \lambda} 
\Big(\int_{\partial B_R} \partial_\nu\phi^{(\alpha)}_k f{\rm d}s(y) + \int_{\partial B_R} \partial_\nu(\Delta\phi^{(\alpha)}_k)g{\rm d}s(y)\Big),\\
\hat{\Lambda}_{\alpha, 1} (\lambda)(f, g) &= \sum_{k\leq E} \partial_\nu\phi^{(\alpha)}_k\Big\vert_{\partial B_R} \frac{1}{\lambda^{(\alpha)}_k - \lambda} 
\Big(\int_{\partial B_R} \partial_\nu\phi^{(\alpha)}_k f{\rm d}s(y) + \int_{\partial B_R} \partial_\nu(\Delta\phi^{(\alpha)}_k)g{\rm d}s(y)\Big),
\end{align*}
and
\begin{align*}
\tilde{\Lambda}_{\alpha, 2} (\lambda)(f, g) &= \sum_{k>E} \partial_\nu(\Delta\phi^{(\alpha)}_k)\Big\vert_{\partial B_R} \frac{1}{\lambda^{(\alpha)}_k - \lambda} 
\Big(\int_{\partial B_R} \partial_\nu\phi^{(\alpha)}_k f{\rm d}s(y) + \int_{\partial B_R} \partial_\nu(\Delta\phi^{(\alpha)}_k)g{\rm d}s(y)\Big),\\
\hat{\Lambda}_{\alpha, 2} (\lambda) (f, g)&= \sum_{k\leq E} \partial_\nu(\Delta\phi^{(\alpha)}_k)\Big\vert_{\partial B_R} \frac{1}{\lambda^{(\alpha)}_k - \lambda} 
\Big(\int_{\partial B_R} \partial_\nu\phi^{(\alpha)}_kf{\rm d}s(y) + \int_{\partial B_R} \partial_\nu(\Delta\phi^{(\alpha)}_k)g{\rm d}s(y)\Big).
\end{align*}

First let us consider the derivatives $\hat{\Lambda}^{(j)}_{\alpha, d}(\lambda)$ for $d = 1, 2$. Since $\lambda_k\lesssim k^{4/n}$ for all $k\geq 1$, we have 
the following estimate when $E^{4/n}\lesssim  \Re \lambda$: 
\begin{align}\label{5.7}
\|\hat{\Lambda}^{(j)}_{\alpha, d}(\lambda)\|_d\leq\frac{1}{(\Re \lambda)^{j +1}}, \quad j\geq 0.
\end{align}
Especially, for some sufficiently large $\zeta_0\geq 1$ depending on $E$, we obtain from \eqref{5.7} that when $\Re\lambda = \mathcal{O}(\zeta^4)$
\begin{align}\label{5.8}
\|\hat{\Lambda}_{\alpha, d} (\lambda)\|_d\lesssim \frac{1}{\zeta^4}, \quad \zeta\geq\zeta_0.
\end{align}
Combing \eqref{5.4} and \eqref{5.8} gives 
\begin{align}\label{5.9}
\|V\|^2_{L^2(B_R)} &\lesssim  \frac{1}{\zeta^{\frac{1}{n}}} + \frac{1}{\zeta^{\frac{1}{n}}} \notag\\
& \quad +  \zeta^{15/2} \Big(\|\tilde{\Lambda}_{1,1}(\lambda) - \tilde{\Lambda}_{2,1}(\lambda)\|^2_1
+  \|\tilde{\Lambda}_{1,2}(\lambda) - \tilde{\Lambda}_{2,2}(\lambda)\|^2_2\Big)\notag\\
&\lesssim \frac{1}{\zeta^{\frac{1}{n}}} +  \zeta^{15/2} \Big(\|\tilde{\Lambda}_{1,1}(\lambda) - \tilde{\Lambda}_{2,1}(\lambda)\|^2_1
+  \|\tilde{\Lambda}_{1,2}(\lambda) - \tilde{\Lambda}_{2,2}(\lambda)\|^2_2\Big).
\end{align}
Using Lemma \ref{ddtn} with $t_1 = 0, t_2 = -2$ and \eqref{5.7}, we have for $d = 1, 2$ that 
\begin{align}\label{5.11}
\|\tilde{\Lambda}_{1,d}^{(j)}(\lambda) - \tilde{\Lambda}_{2,d}^{(j)}(\lambda)\|_d \lesssim \frac{1}{(\Re \lambda)^{j + \sigma}},
\end{align}
where $\lambda\in\mathbb{C}, \Re \lambda\geq 2Q, m> 1 + \frac{n}{4}, 0\leq j \leq m$ and $\sigma = \min\{\sigma_1, \sigma_2\}$. Here the constants $\sigma_1$ and $\sigma_2$
are given in Lemma \ref{ddtn}.

Hereafter we assume $\zeta\gg 1$ and $\Re\lambda\geq 2Q$. Let $T:=\tilde{\lambda} - \lambda$ such that $T>0$.
Following \cite[Proof of Proposition 2.1]{AS},  by Taylor's formula, we have for $d = 1, 2$ that 
\begin{align}\label{5.12}
\tilde{\Lambda}_{\alpha, d}(\lambda) &= \sum_{k = 0}^{m -1} \frac{(\lambda - \tilde{\lambda})^k}{k!} \tilde{\Lambda}^{(k)}_{\alpha, d}(\tilde{\lambda}) + 
\int_0^1 \frac{(1 - s)^m(\lambda - \tilde{\lambda})^m}{(m - 1)!} \tilde{\Lambda}^{(m)}_{\alpha, d} (\tilde{\lambda} + s(\lambda - \tilde{\lambda})) {\rm d}s\notag\\
&:= I_{\alpha, d}(\lambda) + R_{\alpha, d}( \lambda). 
\end{align}
Since $\Re\tilde{\lambda} \geq \Re\lambda + T > T$, an application of \eqref{5.11} leads to
\begin{align}\label{5.13}
\|I_d(V_1, \lambda) - I_d(V_2, \lambda)\|_d\lesssim \frac{1}{T^\sigma}.
\end{align}

Next we study $R_{\alpha, 1}( \lambda), \alpha = 1, 2$. We start with $\tilde{\Lambda}^{(m)}_{\alpha, 1}(\lambda)$ appearing in the integral of $R_{\alpha, 1}(\lambda)$.
We know from Lemma \ref{derivative} that
\begin{align*}
\tilde{\Lambda}^{(m)}_{\alpha, 1}(\lambda)f&= \sum_{k>E} \partial_\nu\phi^{(\alpha)}_k\Big\vert_{\partial B_R} \frac{1}{(\lambda^{(\alpha)}_k - \lambda)^m} \\
&\quad \times\Big(\int_{\partial B_R} \partial_\nu\phi^{(\alpha)}_kf{\rm d}s(y) + \int_{\partial B_R} \partial_\nu(\Delta\phi^{(\alpha)}_k)g{\rm d}s(y)\Big).
\end{align*}

For simplicity we denote $\tilde{\lambda} + s(\lambda - \tilde{\lambda}) = \lambda + (1-s)T$ appearing in $\tilde{\Lambda}^{(m)}_{\alpha, d} (\tilde{\lambda} + s(\lambda - \tilde{\lambda})) $ by $\beta = \beta(s)$. We further let
\[
E_\alpha(\lambda) = \max \{j\geq E ; c\lambda_{j + 1}^{(\alpha)}< \Re\lambda\},
\]
where $c$ is any positive constant satisfying $0<c<1$.
Let $E(\lambda) = \max\{E_1(\lambda), E_2(\lambda)\}$. For sufficiently large $\Re\lambda$, we decompose the series in $\tilde{\Lambda}^{(m)}_{\alpha, 1}(\beta)f$
as a sum of a finite one and an infinite one in the following way: 
\[
\tilde{\Lambda}^{(m)}_{\alpha, 1}(\beta)f = \tilde{\Lambda}^{(m)}_{\alpha, 1, 1}(\beta)f + \tilde{\Lambda}^{(m)}_{\alpha, 1, 2}(\beta)f,
\]
where
\begin{align*}
\tilde{\Lambda}^{(m)}_{\alpha, 1, 1}( \beta)& = \sum_{k=E+1}^{E(\lambda)} \partial_\nu\phi^{(\alpha)}_k\Big\vert_{\partial B_R} \frac{1}{(\lambda^{(\alpha)}_k - \beta)^{m + 1}} \\
&\quad \times\Big(\int_{\partial B_R} \partial_\nu\phi^{(\alpha)}_k f{\rm d}s(y) + \int_{\partial B_R} \partial_\nu(\Delta\phi^{(\alpha)}_k)g{\rm d}s(y)\Big)
\end{align*}
and
\begin{align*}
\tilde{\Lambda}^{(m)}_{\alpha, 1, 2}(\beta) &= \sum_{k>E(\lambda)} \partial_\nu\phi^{(\alpha)}_i\Big\vert_{\partial B_R} \frac{1}{(\lambda^{(\alpha)}_k - \beta)^{m + 1}} \\
&\quad\times \Big(\int_{\partial B_R} \partial_\nu\phi^{(\alpha)}_k f{\rm d}s(y) + \int_{\partial B_R} \partial_\nu(\Delta\phi^{(\alpha)}_k)g{\rm d}s(y)\Big).
\end{align*}
Following \cite[Proof of Proposition 2.1]{AS}, we further make the decomposition
\[
\tilde{\Lambda}^{(m)}_{1, 1, 1}(\beta) - \tilde{\Lambda}^{(m)}_{2, 1, 1}(\beta) = L_1 + L_2 + L_3,
\]
where
\begin{align*}
L_1f &= \sum_{k=E+1}^{E(\lambda)} \partial_\nu\phi^{(1)}_k\Big(\frac{1}{(\lambda^{(1)}_k - \beta)^{m + 1}}  - \frac{1}{(\lambda^{(2)}_k - \beta)^{m + 1}} \Big)\\
&\quad \times\Big(\int_{\partial B_R} \partial_\nu\phi^{(1)}_k f{\rm d}s(y) + \int_{\partial B_R} \partial_\nu(\Delta\phi^{(1)}_k)g{\rm d}s(y)\Big),\\
L_2f&= \sum_{k=E+1}^{E(\lambda)} \frac{ \partial_\nu\phi^{(1)}_k}{(\lambda^{(2)}_k - \beta)^{m + 1}} 
\Big(\int_{\partial B_R} (\partial_\nu\phi^{(1)}_k - \partial_\nu\phi^{(2)}_k) f{\rm d}s(y) \\
&\quad + \int_{\partial B_R} ( \partial_\nu(\Delta\phi^{(1)}_k) - \partial_\nu(\Delta\phi^{(2)}_k)) g{\rm d}s(y)  \Big),\\
L_3f&= \sum_{k=E+1}^{E(\lambda)} \frac{1}{(\lambda^{(2)}_k - \beta)^{m + 1}} 
\Big(\int_{\partial B_R} \partial_\nu\phi^{(2)}_kf{\rm d}s(y) \\
&\quad + \int_{\partial B_R} \partial_\nu(\Delta\phi^{(2)}_k)g{\rm d}s(y)\Big) (\partial_\nu\phi^{(1)}_k - \partial_\nu\phi^{(2)}_k).
\end{align*}

When $\varrho>\frac{8}{n} + 1$, we have from a simple calculation that  
\begin{align*}
\|L_1\|&\lesssim \frac{E(\lambda)^\varrho}{|\Im\lambda|^{m + 2}}\varepsilon_0\Big( \sum_{k=E+1}^{E(\lambda)} k^{-\varrho}\|\partial_\nu\phi^{(1)}_k\|^2_{L^2(\partial B_R)}\\
&\quad+ \sum_{k=E+1}^{E(\lambda)} k^{-\varrho}\|\partial_\nu\phi^{(1)}_k\|_{L^2(\partial B_R)}
\|\partial_\nu(\Delta\phi^{(1)}_k)\|_{L^2(\partial B_R)}\Big).
\end{align*}
We point out that the assumption $\Im\lambda\geq 1$ is useful in deriving the above estimate, since there may not be a uniform gap between adjacent eigenvalues. Using the following Weyl-type inequality \cite{LYZ}: 
\begin{align*}
\|\partial_\nu\phi^{(\alpha)}_k\|_{L^2(\partial B_R)}\lesssim k^{2/n}, \quad \|\partial_\nu(\Delta\phi^{(\alpha)}_k)\|_{L^2(\partial B_R)}\lesssim k^{4/n},
\end{align*}
we get 
\begin{align*}
& \sum_{k=E+1}^{E(\lambda)} k^{-\varrho}\|\partial_\nu\phi^{(1)}_k\|^2_{L^2(\partial B_R)}+ \sum_{k=E+1}^{E(\lambda)} k^{-\varrho}\|\partial_\nu\phi^{(1)}_k\|_{L^2(\partial B_R)}
\|\partial_\nu(\Delta\phi^{(1)}_k)\|_{L^2(\partial B_R)}\\
&\lesssim \sum_{k\geq 1} k^{-\varrho + 6/n}.
\end{align*}
Thus, we obtain the estimate
\[
\|L_1\|\lesssim  \frac{E(\lambda)^\varrho}{|\Im\lambda|^{m + 2}} \varepsilon_0.
\]

Denote the two sets of spectral data discrepancy by
\begin{align*}
 \varepsilon_1 &= \sum_{k\geq 1} k^{-4m/n} \|\partial_\nu\phi^{(1)}_{k + E} - \partial_\nu\phi^{(2)}_{k + E}\|_{L^2(\partial B_R)},\\
 \varepsilon_2 &= \sum_{k\geq 1} k^{-4m/n} \|\partial_\nu(\Delta\phi^{(1)}_{k + E}) - \partial_\nu(\Delta\phi^{(2)}_{k + E})\|_{L^2(\partial B_R)}.
\end{align*}
Similarly, we may show that 
\begin{align*}
\|L_2\| &\lesssim  \frac{E(\lambda)^{4m/n + 2/n}}{|\Im\lambda|^{m + 1}} ( \varepsilon_1 +  \varepsilon_2),\\
\|L_3\| &\lesssim  \frac{E(\lambda)^{4m/n + 4/n}}{|\Im\lambda|^{m + 1}}  \varepsilon_1.
\end{align*}
Letting $ \varepsilon =  \varepsilon_0 +  \varepsilon_1 +  \varepsilon_2$, we have
\[
\|L_1\| + \|L_2\| + \|L_3\| \lesssim  \frac{E(\lambda)^\varrho + E(\lambda)^{4(m + 1)/n}}{|\Im\lambda|^{m + 1}} \varepsilon.
\]
Choosing $\beta = 4(m + 1)/n$ and recalling that $m> 1 + \frac{n}{4}$ we have $\beta>\frac{8}{n} + 1$, which gives
\begin{align}\label{5.14}
\|\tilde{\Lambda}^{(m)}_{1, 1, 1}(\beta) - \tilde{\Lambda}^{(m)}_{2, 1, 1}(\beta)\|_1\lesssim  \frac{E(\lambda)^{4(m + 1)/n}}{|\Im\lambda|^{m + 1}}\varepsilon.
\end{align}

From the inequality
\[
|\lambda^{(\alpha)}_k - \beta|\geq \lambda^{(\alpha)}_k - \Re\lambda + (1 - s)\rho \geq \lambda^{(\alpha)}_k - \Re\lambda \geq (1 - c)\lambda^{(\alpha)}_k,
\]
we obtain 
\[
|\lambda^{(\alpha)}_k - \beta|^{-(m + 1)} \lesssim \frac{1}{({\lambda^{(\alpha)}_k})^{m + 1}}\lesssim \frac{1}{(k^{\frac{4}{n}})^{m + 1}}\lesssim k^{-4m/n}.
\]
Therefore,  using similar arguments above by decomposing $\tilde{\Lambda}^{(m)}_{1, 1, 2}(\lambda) - \tilde{\Lambda}^{(m)}_{2, 1, 2}(\lambda)$ into three parts, we can obtain 
\[
\|\tilde{\Lambda}^{(m)}_{1, 1, 2}(\lambda) - \tilde{\Lambda}^{(m)}_{2, 1, 2}(\lambda)\|_1\lesssim\varepsilon,
\]
which, together with \eqref{5.14}, implies
\[
\|\tilde{\Lambda}^{(m)}_{1,1}(\lambda) - \tilde{\Lambda}^{(m)}_{2,1}(\lambda)\|_1\lesssim E(\lambda)^{4(m + 1)/n}\varepsilon.
\]
From the definition of $E(\lambda)$ we obtain
\[
E^{4/n}\lesssim\lambda^{(\alpha)}_E\lesssim \frac{1}{c}\Re\lambda.
\]
As a consequence, it holds that
\[
\|\tilde{\Lambda}^{(m)}_{1, 1}(\lambda) - \tilde{\Lambda}^{(m)}_{2,1}( \lambda)\|_1\lesssim (\Re\lambda)^{m + 1}\varepsilon,
\]
which, together with $\Re\lambda = \mathcal{O}(\zeta^4)$, gives 
\[
\|\tilde{\Lambda}^{(m)}_{1,1}(\lambda) - \tilde{\Lambda}^{(m)}_{2,1}(\lambda)\|_1\lesssim \zeta^{4(m + 1)}\varepsilon.
\]
Then
\begin{align}\label{5.15}
\|R_{1,1}(\lambda) - R_{1,2}(\lambda)\|_1 \lesssim T^m\zeta^{4(m + 1)}\varepsilon.
\end{align}

Combining \eqref{5.12}, \eqref{5.13} and \eqref{5.15} leads to 
\[
\|\tilde{\Lambda}_{1,1}(\beta) - \tilde{\Lambda}_{2,1}(\beta)\|_1\lesssim \frac{1}{T^\sigma} + T^m\zeta^{4(m + 1)}\varepsilon.
\]
Using similar arguments, we obtain
\begin{align*}
\|R_{1,2}(\lambda) - R_{2,2}(\lambda)\|_2 \lesssim T^m\zeta^{4(m + 1)}\varepsilon,
\end{align*}
which gives
\[
\|\tilde{\Lambda}_{1,2}(\beta) - \tilde{\Lambda}_{2,2}(\beta)\|_2\lesssim \frac{1}{T^\sigma} + T^m\zeta^{4(m + 1)}\varepsilon.
\]
Substituting the above estimates into \eqref{5.9} yields 
\begin{align*}
\|V\|^2_{L^2(B_R)} \lesssim \frac{1}{\zeta^{\frac{1}{n}}} +  \zeta^{15/2}\Big(\frac{1}{T^{2\sigma}} + T^{2m}\zeta^{8(m + 1)}\varepsilon^2\Big).
\end{align*}
Taking $T = (\Re\lambda)^\varsigma$ where $\varsigma = 1/\sigma$ gives 
\begin{align*}
\|V\|^2_{L^2(B_R)} \lesssim \frac{1}{\zeta^{\frac{1}{n}}} + \zeta^{15/2 + 8\varsigma m + 8(m + 1)}\varepsilon^2.
\end{align*}
Using the standard minimization with respect to $\zeta$ (cf. \cite[Proof of Proposition 2.1]{AS}), we obtain the stability estimate
\[
\|V\|^2_{L^2(B_R)} \lesssim \varepsilon^{2\delta},\quad
\delta  = \frac{1}{16n(2 + \varsigma m + m )},
\]
which completes the proof of Theorem \ref{main}. 
\end{proof}

\appendix

\section{Useful estimates}

\begin{theorem}\label{regularity}
Let $u\in H^2(B_R)$ be a weak solution of the following boundary value problem with the Navier boundary condition: 
\begin{align*}
\begin{cases}
H_V u = F_1 \quad &\text{in}\,B_R,\\
u = f \quad  \quad & \text{on}\,\partial B_R,\\
\Delta u = g  \quad & \text{on}\,\partial B_R,
\end{cases}
\end{align*}
where $H_V = \Delta^2 + V$ and $0$ is not an eigenvalue of $H_V$. Then
\[
\|u\|_{H^2(\Omega)} \lesssim \|F\|_{L^2(\Omega)} + \|f\|_{H^{\frac{3}{2}}(\partial\Omega)} + \|g\|_{H^{-\frac{1}{2}}(\partial\Omega)}.
\]
\end{theorem}

The following lemma gives an estimate for the normal derivatives of the eigenfunctions on $\partial B_R$ and a Weyl-type inequality for the Dirichlet eigenvalues.

\begin{lemma}\label{eigenfunction_est1}
The following estimate holds in $\mathbb R^n$:
\begin{align}\label{boundary_estimate_2}
\|\partial_\nu \phi_k\|_{L^2(\partial B_R)}\leq C\lambda_k^{\frac{1}{2}},\quad \|\partial_\nu (\Delta\phi_k)\|_{L^2(\partial B_R)}\leq C\lambda_k,
\end{align}
where the positive constant $C$ is independent of $k$. Moreover, the following Weyl-type inequality holds for the Dirichlet eigenvalues $\{\mu_k\}_{k=1}^\infty$:
\begin{align}\label{weyl_1}
E_1 k^{4/n}\leq \lambda_k\leq E_2 k^{4/n},
\end{align}
where $E_1$ and $E_2$ are two positive constants independent of $k$.
\end{lemma}

\begin{proof}
We begin with the estimate \eqref{boundary_estimate_2} for the eigenfunctions on the boundary. Let $u$ be an eigenfunction  with eigenvalue $\mu$ such that
\begin{align*}
\begin{cases}
H_V u = \lambda u\quad &\text{in}\, B_R,\\
u = \Delta u = 0 \quad &\text{on}\, \partial B_R.
\end{cases}
\end{align*}
Define a differential operator 
\[
A = \frac{1}{2}(x\cdot \nabla+ \nabla\cdot x) = x\cdot\nabla + \frac{n}{2} = |x|\partial_\nu + \frac{n}{2}.
\]
Denote the commutator of two differential operators by $[\cdot, \, \cdot]$ such that $[O_1, O_2] = O_1 O_2 - O_2 O_1$ for two differential operators $O_1$ and $O_2$. Then we have  
\begin{align}\label{com}
[\Delta^k, A] = 2k\Delta^k, \quad k\in\mathbb N^+.
\end{align}
Denote $B = A\Delta$. A simple calculation gives
\begin{align*}
&\int_{B_R} u [H_V, B] u {\rm d}x
= \int_{B_R} \left(u (\Delta^2 + V) (Bu) - u B (\Delta^2 + V) u \right){\rm d}x\\
&= \int_{B_R} (\Delta^2 u + Vu - \lambda u) Bu {\rm d}x+ \int_{\partial B_R}\left( u\partial_\nu(\Delta Bu) - \partial_\nu u\Delta(Bu) \right) {\rm d}s\\
&\quad + \int_{\partial B_R}\left( \Delta u\partial_\nu(Bu) - \partial_\nu (\Delta u)Bu \right) {\rm d}s\\
&= - \int_{\partial B_R} \left(\partial_\nu u\Delta(Bu) + \partial_\nu (\Delta u)Bu\right) {\rm d}s\\
&= - \int_{\partial B_R}\left( \partial_\nu u\Delta(Bu) + R|\partial_\nu (\Delta u)|^2\right){\rm d}s,
\end{align*}
where we have used $u = \Delta u = 0$ on $\partial B_R$ and Green's formula. By \eqref{com}, we have 
\[
\Delta Bu = \Delta A\Delta = (A\Delta + 2\Delta)\Delta = A\Delta^2 + 2\Delta^2. 
\]
It holds that 
\begin{align*}
&\int_{\partial B_R} \partial_\nu u\Delta(Bu){\rm d}s = \int_{\partial B_R} \partial_\nu u (A\Delta^2 + 2\Delta^2)u {\rm d}s \\
&=  \int_{\partial B_R} \Big(\partial_\nu u \big(\big(R\partial_\nu + 
\frac{n}{2}\big)\Delta^2 u\big) + 2\Delta^2u \Big){\rm d}s\\
&= R\int_{\partial B_R} \partial_\nu u \,\partial_\nu(\Delta^2 u){\rm d}s
= R\int_{\partial B_R} \partial_\nu u \,\partial_\nu(\mu u - Vu){\rm d}s,
\end{align*}
where we have used $\Delta^2u = -Vu + \lambda u = 0$ and $u = 0$ on $\partial B_R$. Hence 
\begin{align}\label{I}
 \Big|\int_{\partial B_R} \partial_\nu u\Delta(Bu){\rm d}s\Big| \geq (\lambda - \|V\|_{L^\infty(B_R)}) \int_{\partial B_R} |\partial_\nu u|^2{\rm d}s.
\end{align}

On the other hand we have 
\begin{align}\label{II}
\int_{\partial B_R} \partial_\nu (\Delta u)Bu {\rm d}s= \int_{\partial B_R} \partial_\nu (\Delta u)Bu {\rm d}s = R \int_{\partial B_R} |\partial_\nu (\Delta u)|^2 {\rm d}s.
\end{align}
Moreover, it follows from \eqref{com} that $[H_V, B] = 4\Delta^3 + [V, A\Delta]$, which gives
\begin{align}\label{III}
&\Big|\int_{B_R} u [H_V, B] u {\rm d}x\Big|= \Big|\int_{B_R} \left(4u\Delta^3u + [V, A\Delta]u\right) {\rm d}x\Big|\notag\\
&= \Big|\int_{B_R}\left( 4u\Delta (-Vu + \lambda u) + [V, A\Delta]u \right){\rm d}x\Big|\notag\\
&\leq C\lambda \|u\|^2_{H^2(B_R)}\leq C\lambda^2.
\end{align}
Here we have used the fact that the commutator $[V, A\Delta]$ has order of 2 at most. Using \eqref{I}--\eqref{III} we obtain 
\[
\|\partial_\nu u\|^2_{L^2(\partial B_R)}\leq\lambda,\quad \|\partial_\nu (\Delta u)\|^2_{L^2(\partial B_R)}\leq\lambda^2,
\]
which completes the proof of \eqref{boundary_estimate_2}.

Next, we prove the Weyl-type inequality \eqref{weyl_1}. Assume $\lambda_1<\lambda_2<\cdots$ are the eigenvalues of the
operator $H$. Denote the functional space
\[
H_\vartheta^{2}(B_R)=\{\psi\in H^2(B_R);\, \Delta\psi = \psi =0\text{ on }\partial B_R\},
\]

Then we have following min-max principle: 
\[
\lambda_k=\max_{\phi_1,\cdots,\phi_{k-1}}\min_{\psi\in[\phi_1,\cdots,
\phi_{k-1}]^\perp\atop \psi\in H_\vartheta^{2}(B_R)}\frac{\int_{B_R} |\Delta
\psi|^2 + V|\psi|^2{\rm d}x}{\int_{B_R}\psi^2{\rm d}x}.
\]
Assume that $\lambda_1^{(1)}<\lambda_2^{(1)}<\cdots$ are the eigenvalues for the operator
$\Delta^2$. By the min-max principle, we have
\[
C_1\lambda_k^{(1)}<\lambda_k<C_2\lambda_k^{(1)}, \quad k=1, 2, \dots,
\]
where $C_1$ and $C_2$ are two positive constants depending on $\|V\|_{L^\infty(B_R)}$. We have from Weyl's law \cite{Weyl} for $\Delta^2$ that 
\[
\lim_{k\rightarrow+\infty}\frac{\lambda_k^{(1)}}{k^{4/n}}=D,
\]
where $D$ is a constant. Therefore there exist two constants $E_1$ and $E_2$
such that 
\[
E_1 k^{4/n}\leq \lambda_k\leq E_2 k^{4/n},
\]
which completes the proof.
\end{proof}

Denote the resolvent by $R_V(\lambda) = (-\Delta + V - \lambda)^{-1}, \lambda\in\mathbb C$. The following theorem gives a resonance-free region and a resolvent estimate of $\rho R_V(\lambda)\rho: L^2(\mathbb R^n)\rightarrow H^4(\mathbb R^n)$ for a given $\rho\in C_0^\infty(\mathbb R^n)$ when $n\geq 3$ is odd. 

\begin{theorem}\label{bound_2}
Let $V(x)\in L^\infty_{\rm comp}(\mathbb R^n, \mathbb C)$ and $n\geq 3$ be odd. Then for any given $\rho\in C_0^\infty(\mathbb R^n)$ satisfying $\rho V = V$, i.e., $\text{supp} (V)\subset \text{supp} (\rho)\subset\subset B_R$, there exists a positive constant $C$ depending on $\rho$ and $V$ such that
\begin{align}\label{bound_3}
\|\rho R_V(\lambda)\rho\|_{L^2(B_R)\rightarrow H^j(B_R)}\leq C|\lambda|^{\frac{-2 + j}{4}} \big(e^{2R(\Re\sqrt[4]{\lambda})_-}+ 
e^{2R(\Im\sqrt[4]{\lambda})_-}\big),\quad j = 0, 1, 2, 3, 4,
\end{align}
where $\lambda\in \Omega_\delta$. Here $\Omega_\delta$ denotes the resonance-free region defined as
\begin{align*}
\Omega_\delta:=\Big\{\lambda: {\Im}\sqrt[4]{\lambda}&\geq - A - \delta {\rm log}(1 + |\lambda|^{1/4}), \,{\Re}\sqrt[4]{\lambda} \geq - A - \delta {\rm log}(1 + |\lambda|^{1/4}), |\lambda|^{1/4}\geq C_0\Big\},
\end{align*}
where $A$ and $C_0$ are two positive constants and $\delta$ satisfies $0<\delta<\frac{1}{2R}$. 
\end{theorem}

\begin{proof}

Denote the free resolvent by $R_0(\lambda) =  (-\Delta  - \lambda)^{-1}, \lambda\in\mathbb C.$ Using the following identity
\begin{align}\label{decom}
R_0(\lambda) = (\Delta^2 - \lambda)^{-1} = \frac{1}{2\sqrt{\lambda}}[ (-\Delta - \sqrt{\lambda})^{-1} - (-\Delta + \sqrt{\lambda})^{-1} ]
\end{align}
and \cite[Theorem 3.1]{Dyatlov}, we can prove that when $n\geq 3$ is odd, for each $\rho\in C_0^\infty(\mathbb R^n)$ with ${\rm supp}(\rho)\subset B_R$ and $\lambda\neq 0$
\begin{align}\label{free}
\|\rho R_0(\lambda) \rho\|_{L^2(B_R)\rightarrow L^2(B_R)}\lesssim\frac{1}{\sqrt{\lambda}}\big(e^{2R (\Im\sqrt[4]{\lambda})_-} + e^{2R (\Re\sqrt[4]{\lambda})_-}\big),
\end{align}
where $t_{-}:=\max\{-t,0\}$. Consequently, using \eqref{free} and similar arguments as in the proofs of \cite[Theorem 3.3]{LYZ} we can prove the estimate \eqref{bound_3}.
\end{proof}

\begin{remark}
We discuss the resolvent estimates in even dimensions $n\geq 2$. Since the free resolvent $G_0(\lambda) = (-\Delta - \lambda)^{-1}$ in even dimensions is a convolution operator with
the kernel (see e.g. \cite{FY06})
\begin{align*}
G_0(\lambda) =  \frac{c_n e^{{\rm i} \sqrt{\lambda} |x|}}{|x|^{n - 2}} \int_0^\infty e^{-t} t^{\frac{n - 3}{2}} \Big( \frac{t}{2} - {\rm i}\sqrt{\lambda} |x| \Big)^{\frac{n - 3}{2}} {\rm d}t,
\end{align*}
where $c_n$ is a positive constant depending on the dimension $n$. Then by \eqref{decom} and a direct calculation we have 
\begin{align*}
|G_0(\lambda)| \lesssim \frac{|\lambda|^{\frac{n - 3}{4}}}{|\lambda|} \big(e^{(\Im\sqrt{\lambda})_-|x|} + e^{(\Re\sqrt{\lambda})_-|x|}\big) \lesssim \frac{1}{|\lambda|^{1 - \frac{n - 3}{4}}} e^{(\Im\sqrt{\lambda})_-|x|},
\end{align*}
which implies that only when $1 - \frac{n - 3}{4}>0$, by repeating the above arguments we may obtain similar resolvent estimates for even dimensional cases.
\end{remark}

\end{document}